\documentclass[11pt,a4paper]{article}
\usepackage{amsfonts, amssymb, amsmath, fancyhdr, url, amsthm}
\usepackage{enumerate}
\usepackage{xcolor}
\usepackage{geometry}
\usepackage{tocloft}
\usepackage[backref=page]{hyperref}
\renewcommand*{\backref}[1]{}
\renewcommand*{\backrefalt}[4]{%
	\ifcase #1 (Not cited.)%
	\or        (Cited on page~#2.)%
	\else      (Cited on pages~#2.)%
	\fi}

\hypersetup{
	colorlinks=true,
	linkcolor=blue,
	filecolor=blue,
	citecolor = blue,
	urlcolor=cyan,
}

\newcommand{\K}{K\"ahler}

\newcommand{\ie}{{\em i.e. }}
\newcommand{\eg}{{\em e.g. }}

\newcommand{\RR}{\mathbb{R}}
\newcommand{\ZZ}{\mathbb{Z}}

\def\eqref#1{(\ref{#1})}

\newcommand{\6}{\partial}
\def\1{\sqrt{-1}\:}

\newcommand{\cntrct}                % contraction with a vector field
{\hspace{2pt}\raisebox{1pt}{\text{$\lrcorner$}}\hspace{2pt}}

\newcommand{\cac}{\mathcal{C}}
\newcommand{\caf}{\mathcal{F}}
\newcommand{\cat}{\mathcal{T}}
\newcommand{\ii}{\mathrm{i}}

\renewcommand{\bar}{\overline}
\renewcommand{\phi}{\varphi}
\renewcommand{\epsilon}{\varepsilon}

\newcommand{\Ric}{\operatorname{Ric}}

\renewcommand{\Im}{\operatorname{Im}}

\numberwithin{equation}{section}

\makeatletter

\@addtoreset{equation}{section}
\@addtoreset{footnote}{section}
\makeatother

\def\blacksquare{\hbox{\vrule width 5pt height 5pt depth 0pt}}
\def\endproof{\blacksquare}

\newcommand{\samethanks}[1][\value{footnote}]{\footnotemark[#1]}

\newcounter{Mycounter}[section]
\newcounter{lemma}[section]
\newcounter{claim}[section]

\newcounter{sublemma}[section]

\newcounter{corollary}[section]

\newcounter{theorem}[section]

\newcounter{conjecture}[section]

\newcounter{proposition}[section]

\newcounter{definition}[section]

\newcounter{example}[section]

\newcounter{remark}[section]

\newcounter{problem}[section]

\newcounter{question}[section]

\begin{document}

\begin{center}
{\LARGE\bf
Vaisman manifolds and transversally {\K}-Einstein metrics}\\[3mm]
%%%%%%%%%%%%%%%%%%%%%%%%%%%%%%%%%%%%%%%%%%%%%%%%%%%%%%%%%%%%
{\large\bf
Vladimir Slesar\footnote{Partially supported by Romanian Ministry of Education and Research, Program PN-III, Project number PN-III-P4-ID-PCE-2020-0025, Contract  30/04.02.2021} and Gabriel-Eduard V\^ilcu\samethanks{}

}
\end{center}

\hfill

{\small
\hspace{0.15\linewidth}
\begin{minipage}[t]{0.7\linewidth}
{\bf Abstract}
We use the transverse {\K}-Ricci flow on the canonical foliation of a closed Vaisman manifold to deform the Vaisman metric into another Vaisman metric  with a transverse {\K}-Einstein structure. We also study the main features of such a manifold. Among other results, using techniques from the theory of parabolic equations, we obtain a direct proof for the short time existence of the solution for transverse {\K}-Ricci flow on Vaisman manifolds, recovering in a particular setting a result of Bedulli, He and Vezzoni [J. Geom. Anal. 28, 697--725 (2018)], but without employing the Molino structure theorem. Moreover, we investigate Einstein-Weyl structures in the setting of Vaisman manifolds and find their relationship with quasi-Einstein metrics. Some examples are also provided to illustrate the main results. \\
\end{minipage}
}
% \marginpar{De rescris}

\hspace{5mm}

\noindent{\bf Keywords: Vaisman manifold, transverse {\K}-Ricci flow, transverse {\K}-Einstein metric, quasi-Einstein metric}
%\noindent {\bf 2000 Mathematics Subject Classification:} { }\\[4mm]

%%%%%%%%%%%%%%%%%%%%%%%%%%%%%%%%%%%%%%%%%%%%%%%%

\tableofcontents

\date{\today }

\section{Introduction}

Let us consider a Hermitian manifold $(M, J, g)$ and $\omega(\cdot, \cdot):=g(\cdot, J\cdot)$ the corresponding fundamental 2-form. We assume throughout the paper that the manifold $M$ is closed, connected, with $\dim_\RR M\ge 2$.

If there is a nonzero 1-form $\theta$ such that $d\theta=\theta\wedge \omega$, with $\theta$ parallel with respect to the Levi-Civita connection of the metric $g$, then the manifold is said to be a {\em Vaisman manifold}. Moreover, the 1-form $\theta$ is called the \emph{Lee form}.

Being parallel, the Lee form is also closed, and Vaisman manifolds are locally conformally {\K} (LCK). We refer to \cite{Dr-Or, Or-Ve3} for definition and properties of LCK manifolds.  Note that most part of known LCK manifolds are in fact Vaisman. Actually, it is known that all compact homogeneous LCK manifolds are Vaisman, according to a result proved in \cite{Ga-Mo-Or}.

Many significant results on the geometry of Vaisman manifolds have been obtained over time (see, e.g., the recent works \cite{Al-Ha-Ka, Ca-Ni-Ma-Yu2, Ga-Mo-Or, In-Le, Ma-Mo-Pi, Or-Ve2, Pi} and the references therein). A remarkable feature of the Vaisman manifolds is the canonical foliated structure \cite{Ch-Pi, Ts}. The leaves of the foliations are two-dimensional, local Euclidean, while the transverse part of the two-fundamental form, denoted in this paper by $\omega^T_0$ (see Section \ref{transverse_K_R_flow}), is in fact {\K} \cite{Va}. Note that it is a commune feature that Vaisman manifolds share with Sasaki manifolds. We would also like to point out that the relation between Vaisman and Sasakian manifolds via the mapping torus construction was explored in \cite{Ba-Ma-Op}. It was obtained that a compact Vaisman manifold is finitely covered by the product of a compact Sasakian manifold and a circle. This nice result gives topological information about compact Vaisman manifolds. We also note that a Hard Lefschetz Theorem for the de Rham cohomology and for the basic cohomology with respect to the Lee vector field on compact Vaisman manifolds was proved in \cite{Ca-Ni-Ma-Yu1}. The proof of this theorem uses some results regarding the basic cohomologies of compact oriented Riemannian manifolds with respect to some Riemannian foliations.

As the transverse metric has a {\K} structure, it is natural to investigate if we could find Vaisman metrics with transverse {\K}-Einstein structure.

In the setting of compact Sasaki manifolds, under assumption that the basic Chern class $c_b^1(M)=k[\omega^T_0]$ (see Section \ref{prelim} for the definition of basic Chern class), with $k=0$ or $k=-1$, in \cite{Sm-Wa-Zh} the authors present a method for obtaining Sasaki metrics which are transversally {\K}-Einstein. This notion is equivalent to $\eta$-Einstein metrics (see \cite{Bo-Ga-Ma} for definition and properties). More precisely, under assumption that the basic Chern class $c_b^1(M)=k[\omega^T_0]$, %(see Section \ref{prelim} for the definition of basic Chern class),
with $k=0$ or $k=-1$, they use a deformation of the Sasaki structure based on a transverse {\K}-Ricci flow (called Sasaki-Ricci flow). Note that in \cite{Wa-Zh}, for a Sasaki manifold of dimension 3 with  $c_b^1(M)>0$ it was proven that a Sasaki-Ricci soliton can be obtained.

It is thus natural to ask if on a Vaisman manifold (which is also endowed with a transverse {\K} structure, but has a foliated structure of dimension 2, more complicated than a Sasakian Riemannian flow) a transverse {\K}-Ricci flow can be used in a similar manner, obtaining thus a counterpart of $\eta$-Einstein manifolds.

Note that in \cite{El Ka}, using the Molino structure theorem for Riemannian foliations \cite{Mo}, the author devised a Calabi-Yau theory for the transverse geometry of a foliation using the continuity method. Using this theory, in \cite{Be-He-Ve} the short time solution was obtained and long time behavior of transverse {\K}-Ricci flow was investigated.
In the first part of this paper we show that for the particular setting of Vaisman manifolds a direct proof for the short time existence of the solution of the {\K}-Ricci flow. This proof is based on a idea developed in \cite{Sm-Wa-Zh}, using only standard results from the theory of parabolic equations, thus without employing the structure theorem of Molino.

We should also check that the transverse {\K}-Ricci flow preserve other structures (such as Sasakian or Vaisman), and as remarked in \cite[Section 3]{Sm-Wa-Zh}, this is not possible by direct computation. Also, we should emphasize the fact that, in general, the class of Vaisman structures is not stable under small deformations \cite{Be}.

Similar to \cite{Sm-Wa-Zh}, to solve this problem, for the particular case when $c_b^1(M)=-[\omega^T_0]$, we associate to the flow a compatible deformation of the Vaisman structure. More precisely, we deform also the complex structure of the manifold, fixing the Lee form.  As a consequence we obtain a new Vaisman structure on the manifold with a {\K}-Einstein transverse metric.
We also hope that this approach will be useful in the context of an increased interest in the study of Chern-Ricci type flows on non-{\K} manifolds (see \eg \cite{To-We1, Gi-Sm}).

In the final part of the paper we investigate several properties of the Vaisman manifolds with transverse {\K}-Einstein metrics and obtain the relation between Einstein-Weyl \cite{Or-Ve0} and quasi-Einstein Vaisman metrics \cite{Go-Va}.
Note that quasi-Einstein metrics arose naturally in the study of exact solutions of the Einstein field equations and highly relevant examples are provided by the Robertson-Walker spacetimes \cite{Ch}. In particular, quasi-Einstein Vaisman metrics were investigated by several authors and some relevant results concerning their geometry can be find in \cite{Dr2,Go-Va}. It is worth mentioning that quasi-Einstein metrics are abundant in LCK setting (see \eg \cite[page 65]{Dr1}).
In particular, one of the main subfamily of Vaisman manifolds, denoted by $\mathcal{P}_0K$ (see \cite{Va0} for details), consists of manifolds that naturally possess quasi-Einstein metrics. It is important to emphasize that the are various generalizations of quasi-Einstein metrics, such as mixed quasi-Einstein, nearly quasi-Einstein,  $m$-quasi-Einstein and generalized quasi-Einstein metrics, these being also of prime interest in some theories of modern physics (see \eg \cite{Br-Ga-Gi-Va,Ca,Ch,De-Sh,De,Fr-An-Te,Hu-Li-Zh}). It is known that examples of generalized quasi-Einstein metrics can be generated through conformal transformations and warped product extensions of some quasi-Einstein metrics.

The paper is organized as follows. In Section \ref{prelim} we outline the basic properties of Vaisman manifolds and the possibility to deform such manifolds using a differential basic function. In Section \ref{transverse_K_R_flow} we introduce the transverse {\K}-Ricci flow and prove the short time existence. %In Section \ref{max_time_exist}  we prove the short time solution of the  transverse {\K}-Ricci flow on Vaisman manifolds.
In Section \ref{eta} we investigate quasi-Einstein Vaisman metrics and find their relationship with Einstein-Weyl structures.
%In the final part we also present examples of Vaisman manifolda with a transverse {\K}-Einstein metric.
Moreover, some illustrative examples are given along the paper.

\section{Preliminaries}\label{prelim}

\subsection{Vaisman manifolds}\label{Vaisman_manifolds}

In the following we present the definition and the main properties of Vaisman manifolds. For a more detailed exposure we indicate \cite{Dr-Or, Or-Ve1}.

Assume that $(M,J)$ is a complex manifold which is connected and closed, of real dimension $2n+2$, with $n\ge 1$. Let $g$ be a Hermitian metric on $M$. We define the fundamental 2-form $\omega(X,Y):=g(X,JY)$, with $X, Y \in \Gamma(TM)$. By abuse, we sometimes call ``metric'' the fundamental form $\omega$. We also denote by $\nabla^\omega$ the Levi-Civita connection of $\omega$.

\begin{definition}
	The manifold $(M,J,g)$ is called {\em Vaisman} if there exists a 1-form $\theta$, parallel with respect to $\nabla^\omega$, which satisfies the condition $d\omega=\theta\wedge\omega$.
\end{definition}	

\begin{remark}
The 1-form $\theta$ is closed, as it is parallel, and is called the {\em Lee form} \cite{Dr-Or}.
\end{remark}

\medskip

We can also normalize the Lee form such that  $\left\Vert \theta \right\Vert =1$. We denote by $\theta ^{c}:=\theta \circ J$ the {\em anti-Lee form}. Let $U=\theta ^{\sharp}$ and $V=\theta ^{c\sharp}$ denote the $g$-dual vector fields, thus

\[
\theta(U)=1,\quad \theta^c(V)=1,\quad \theta(V)=0,\quad \theta^c(U)=0.
\]

The fundamental 2-form is determined by the Lee form and the complex structure. We have
\begin{equation} \label{theta_theta_c}
		\omega=d \theta^c- \theta\wedge \theta^c,
\end{equation}
where $d$ is the de\thinspace Rham differential operator.
This above relation is due to I. Vaisman \cite{Va}.

\begin{remark}\label{K_covering}
The link between Sasaki and Vaisman structures is established by the fact that the universal Riemannian cover of a Vaisman manifold is a metric cone over a Sasaki manifold \cite{Or-Ve}. Thus, locally, on a Vaisman manifold we have a local Sasaki structure transverse to the flow determined by $U$.
\end{remark}

The fact that $\theta$ is parallel implies \cite{Va}:
\begin{equation}\label{conex_Vaisman}
\nabla^\omega _UU=\nabla^\omega _VU= \nabla^\omega _UV=\nabla^\omega _VV=0.
\end{equation}

\subsection{The foliated structure on a Vaisman manifold}\label{canonical_foliation}

In particular, $U$ and $V$ generate a foliation $\caf$ with leafwise dimension $2$ (called the \emph{canonical foliation} \cite{Ch-Pi, Ts}). From \eqref{conex_Vaisman}, we see that the leaves are minimal submanifolds.

As this foliated structure plays an important role in our further considerations, we describe next the main properties.

A metric defined on a foliated manifold is said to be bundle-like if locally it can be identified with a Riemannian submersion \cite{Re}. In the case of a Vaisman manifold, the metric falls in this category with respect to the canonical foliation \cite{Va}.

Consider the foliated distribution $\caf$ generated by the vector fields $U$ and $V$. If  $Q$ is the transverse orthogonal complement of $\caf$ with respect to $g$, then we have a canonical splitting of the tangent bundle

\begin{equation}\label{split_tang}
TM=Q\oplus T\caf.
\end{equation}
 This also imply a splitting of the metric
\[
g=g^T\oplus g^{T\caf},
\]
where $g^T$ is the transverse metric, $g^{T\caf}$ being the leafwise metric. The later is in fact Euclidean \cite{Va}, while the transverse metric, as in the case of a Riemannian submersion, can be locally projected on a local {\em transversal}. Here, by a local transversal we mean a local submanifold of dimension $2n$, transverse to the leaves.

Other geometric objects that can be projected locally on a local transversal are the {basic differential forms}.

\begin{definition}
The de\thinspace Rham complex of the basic (projectable) differential
forms is defined as

\[
\Omega _b\left( M\right) :=\left\{ \alpha \in \Omega \left(
M\right) \mid \iota _X\alpha =0,\,\,\mathrm{Lie}_X\alpha =0\mbox{\,\, for any\,\,}X\in \Gamma \left(T\caf \right) \right\} .
\]
By $\mathrm{Lie}_X$ we denote the Lie derivative along $X$, while $\iota $ stands for interior product.
The \emph{basic de\thinspace Rham derivative} is defined as $d_b:=d_{\mid \Omega _b\left( M\right)}$
(see \emph{e.g.}  \cite{Ton}).
\end{definition}

On a Vaisman manifold the complex structure $J$ is also projectable, as well as the transverse fundamental two form $\omega^T$ (which is thus a basic form, according to the above definition). The canonical foliation is endowed with a transverse {\K} structure \cite{Va}.
Thus we can consider a local orthonormal frame $\{e_{1},\ldots,e_n,e_{n+1,}\ldots,e_{2n},U,V \}$, such that

\[
%\begin{split}
J(U)=V, \qquad J(V)=-U, \quad J(e_i)=e_{n+i}, \quad J(e_{n+i})=-e_i.
%\end{split}
\]

\begin{definition}
A convenient connection for the transverse geometry of the foliation is represented by the {\em Bott connection} $\nabla^T$ (see \cite[Chapter 3]{Ton}) defined by
\[
\begin{cases}
\nabla^T_X Y:=\pi_Q(\nabla^T_X Y)\,\,\mbox{if}\,\, X\in \Gamma(Q),\\
\nabla^T_X Y:=\pi_Q([X, Y])\,\,\mbox{if}\,\, X\in \Gamma(T\caf).
\end{cases}
\]
\end{definition}
\begin{remark}\label{curvature}
It is well known that the restriction of the above connection
\[
\nabla^T:\Gamma(Q)\times \Gamma(Q)\rightarrow \Gamma(Q)
\]
as well as the \emph{transverse Riemann curvature tensor filed}
\[
R^T: \Gamma(Q)\times\Gamma(Q)\times\Gamma(Q)\times\Gamma(Q)\rightarrow \Omega_b^1(M)
\]
are projectable on a local transversal \cite[Chapter 3]{Ton}.
\end{remark}

The transverse Ricci tensor is defined using the transverse Riemann curvature tensor and the above local orthonormal frame.
\begin{definition}
For any $X, Y \in \Gamma(Q)$, the \emph{transverse Ricci tensor filed} is defined as
\[
\Ric^T(X,Y)=\sum_{i=1}^{2n} R^T(e_i,X,Y,e_i).
\]
The \emph{transverse Ricci operator} $\Ric^T:\Gamma(Q)\rightarrow \Gamma(Q)$ is defined such that
\[
\Ric^T(X,Y)=g^T(\Ric^T(X),Y).
\]
\end{definition}

\begin{remark}
The transverse Ricci operator defined above can be naturally extended on the complex of basic differential forms, obtaining an operator $\Ric^T: \Omega _b\left( M\right)\rightarrow \Omega _b\left( M\right)$.
\end{remark}
In the next subsection we describe more explicitly the above projectable mathematical objects using foliated coordinates.

\subsection{Vaisman structure described in a foliated chart}\label{foliated_map}
We consider  local foliated charts of the form $\left(\mathcal{U}, z^j, \bar{z}^j,x,y\right)$, $\mathcal{U}$ being homeomorphic to the product $\mathcal{T}\times \mathcal{O}$, where for the local transversal $\mathcal{T}$ we can be considered a local chart $\left(\mathcal{T}, z^j, \bar{z}^j\right)$ on a complex {\K} manifold, while $(\mathcal{O},x,y)$ is a local chart on the Euclidean space $\RR^2$.
$z^j$, $\bar{z}^j$ are the transverse complex coordinates, $z^j=x^j+\ii y^j$, $\bar{z}^j=x^j-\ii y^j$  while $x,y$ are the leafwise real coordinates. More precisely, we choose the coordinates such that

\[
%\begin{split}
\frac \6 {\6x}=U,\quad \frac \6 {\6y}=V,
%\end{split}
\]
where $U$ is the Lee vector field, $V$ being the anti-Lee vector field.

\begin{remark}\label{basic_vector}
Considering the complex version of the basic de\thinspace Rham complex $\Omega_b(M)$, we see that the differential forms $dz^j$, $d\bar z^j$ are basic. The dual vector fields $\frac \6 {\6z^j}$, $\frac \6 {\6\bar z^j}$ are \emph{basic vector fields}. A defining relation for a basic vector field $X\in \Gamma (Q)$ is $[X, Y]\in \Gamma (Q)$, for any leafwise vector field $Y \in \Gamma(T\caf)$ (see \eg \cite[Chapter 1]{Mo}).
\end{remark}

As a Vaisman manifold has a transverse {\K} structure, we consider $h$ a local transverse {\K} potential, and describe the local structure of a Vaisman manifold using $h$. Note that the local function $h$ is basic. First we express the complex structure $J$ using the above local foliated coordinates.

In the local map, the standard complex structure $J_0$ is defined by the relations
\begin{equation*}
%\begin{split}
J_0\left( \frac\partial{\partial x^j} \right)=\frac\partial{\partial y^j}, \quad
J_0\left( \frac\partial{\partial y^j} \right)=-\frac\partial{\partial x^j},\quad
J_0\left( \frac\partial{\partial x} \right)=\frac\partial{\partial y}, \quad
J_0\left( \frac\partial{\partial y} \right)=-\frac\partial{\partial x}.
%\end{split}
\end{equation*}
We emphasized in the \ref{K_covering}, that a Vaisman manifold can be described locally as a Riemannian submersion with the fibers having a local structure of a Sasaki manifold (see also \cite{Or-Ve}). For the local transverse {\K} potential  in the case of a Sasaki manifold we refer to \cite{Go-Ko-Nu, Sm-Wa-Zh}. Thus the complex structure $J$ is expressed as
\[
\begin{split}
  J&=J_0+\frac\partial{\partial x}\otimes d^c \left(h\right)+\frac\partial{\partial y}\otimes d^c \left(h\right)\circ J_0 \\
   &=J_0+\frac\partial{\partial x}\otimes (\ii \frac{\6 h}{\6 \bar{z}^j}d\bar{z}^j-\ii \frac{\6 h}{\6 z^j}dz^j)+\frac\partial{\partial y}\otimes (\ii \frac{\6 h}{\6 \bar{z}^j}d\bar{z}^j-\ii \frac{\6 h}{\6 z^j}dz^j)\circ J_0.
\end{split}
\]

The Lee and anti-Lee forms are
\[
\theta=dx, \quad \theta^c =-\theta \circ J=dx+ \ii \frac{\6 h}{\6z^j}dz^j-\ii \frac{\6 h}{\6\bar{z}^j}d\bar{z}^j.
\]

We introduce the basis of local vector fields $\{X_1, \ldots  X_n, \bar{X}_1, \ldots  \bar{X}_n\}$ on $Q$, where
\begin{equation}\label{X_j}
%\begin{split}
 X_j =\frac{\6}{\6 z^j}-\ii \frac{\6 h}{\6 z^j}\frac{\6}{\6 x}, \quad
 \bar{X}_j =\frac{\6}{\6 z^j}+\ii \frac{\6 h}{\6 \bar{z}^j}\frac{\6}{\6 x}.
%\end{split}
\end{equation}
We can check by a direct computation that
\begin{equation}\label{J_X}
\begin{split}
J(X_j)=\ii X_j,\quad J(\bar X_j)=-\ii \bar X_j.
\end{split}
\end{equation}

The dual of the basis $\{X_i,\bar{X}_i, U, V\}$ is in fact $\{dz^i, d\bar{z}^i, \theta, \theta^c\}$. Notice that the local 1-forms $dz^i, d\bar{z}^i$ are basic, and thus projectable.

We also denote
\[
g(X_j, \bar{X}_k)=g^T_{j\bar{k}}.
\]

In local coordinates $g^T$ is written as below.

\begin{equation*}
    g^T=g_{j\bar k}dz^j \odot d\bar z^k:=g_{j\bar k}(dz^j \otimes d\bar z^k
    +d\bar z^k \otimes dz^j).
\end{equation*}

The associated  transverse fundamental two-form $\omega^T$ is written in the following manner.

\begin{equation*}
    \omega^T=\ii g_{j\bar k}dz^j \wedge \bar z^k.
\end{equation*}

We compute now the coefficients of the Bott connection in the local foliated coordinates. First, the following Lie brackets vanish

\begin{equation*}
    [X_j,X_k]=[\bar X_j,\bar X_k]=[U,X_j]=[U,\bar X_j]=0.
\end{equation*}
For the anti Lee vector filed $V$ we get a similar result. Using this result, similar to the {\K} case, we see that the only Christoffel coefficients of $\nabla^T$ that do not vanish are $\Gamma^k_{jl}$ and $\Gamma^{\bar k}_{\bar j \bar l}$ (see also \cite[Section 4]{Sm-Wa-Zh}). By a direct computation we get

\begin{equation*}
    \Gamma^k_{jl}=(g^T)^{k\bar m}\frac{\6g^T_{l\bar m}}{\6 z^j},
\end{equation*}
$(g^T)^{k\bar m}$ being the entries of the inverse of the matrix $g^T_{j\bar k}$. As both Sasaki manifolds and Vaisman manifolds have a transverse {\K} structure, the results and computations are similar to \cite[Section 4]{Sm-Wa-Zh} (see also \cite{Sl-Vi-Vil}).

The Bott connection $\nabla^T$ defined in the previous part can be naturally extended on basic differential forms, so the differential operators $\6_b:\Omega_b(M)\rightarrow \Omega_b(M)$ and $\bar \6_b:\Omega_b(M)\rightarrow \Omega_b(M)$ can be defined as
\begin{equation}\label{del_b}
   % \begin{split}
    \6_b=\sum_{i=1}^ndz^i \wedge \nabla^T_{X_i}, \quad
    \bar{\6}_b=\sum_{i=1}^nd\bar{z}^i \wedge \nabla^T_{\bar{X}_i}, \\
   % \end{split}
\end{equation}

\begin{remark}\label{del_b_nonbasic}
In the next section we apply the above derivatives on smooth functions which are not necessarily basic. If $f $ varies along the leaves, then using formulas \eqref{del_b}, we can naturally extend the operators $\6_b$, $\bar \6_b$ on $f$. Thus $\ii \6_b\bar\6_b f$ is a real differential 2-form which vanishes when applied on the leafwise distribution $T\caf$ but it is not necessarily basic (\ie invariant along leaves).
\end{remark}

We pass to the transverse Ricci operator. Assume that

\begin{equation*}\label{oper_Ricci}
    \Ric^T(\omega^T)=R_{j\bar k} dz^j \wedge d\bar z^k.
\end{equation*}
Then, as in the case of {\K} manifolds (see \eg \cite[Section 2.1]{To}), we obtain

\begin{equation}\label{coef_Ricci}
    R^T_{j\bar k}=-\frac{\6^2}{\6z^j \6 \bar z^{k}}\log \det (g^T_{j \bar k}).
\end{equation}

From \eqref{oper_Ricci} and \eqref{coef_Ricci}, if $\omega^T$ and $\tilde \omega^T$ are two different transverse fundamental forms, then
\begin{equation}\label{Ric_Chern}
    \Ric^T(\omega^T)-\Ric^T(\tilde\omega^T)=\ii \6_b \bar \6_b \log \frac{\det \tilde g}{\det g}.
\end{equation}

Consider now the following class of the basic Dolbeaut cohomology w.r.t. the complex structure of the manifold.
\begin{equation*}
    \begin{split}
    H_b^{1,1}(M)&:=\frac{\{\bar \6_b\--\mbox{closed real (1,1) basic forms}\}}{\Im \bar \6_b} \\
    &=\frac{\{\bar \6_b\--\mbox{closed real (1,1) basic forms}\}}{\Im \6_b\bar \6_b}.
    \end{split}
\end{equation*}
The last equality is based on the basic $\6\bar \6$\--Lemma proved in \cite[Proposition 3.5.1]{El Ka}.

As the function in \eqref{Ric_Chern} is globally defined and in a local map has the value $\log \frac{\det (\tilde g_{j \bar k})}{\det (g_{j \bar k})}$, we see that the basic Chern class $c_b^1(M)\in H_b^{1,1}(M)$ defined as
\begin{equation*}
    c_b^1(M):=[\Ric^T(\omega^T)]
\end{equation*}
does not depend on the metric (see also \cite{El Ka}).

Similar to \cite{Sm-Wa-Zh}, for a smooth function (not necessarily basic) we denote
\begin{equation}\label{derivate}
    \begin{split}
    f_{\mid j}&:=X_j(f),\quad f_{\mid \bar{j}}:=\bar{X}_j(f), \\
     f_{\mid j \bar{k}}&:=\nabla^g df (X_j,\bar{X}_k)=X_j df(\bar{X}_k)-df (\nabla^g_{X_j}\bar{X}_k).
    \end{split}
\end{equation}
where $\nabla^g df$ is the Hesse tensor field associated to the function $f$, $\nabla^g$ being the Levi-Civita linear connection of the initial metric $g\equiv g_0$.

\begin{remark}\label{leading_symbol}
If we change the metric, from \eqref{derivate} we see that the Hesse tensor field differs by an differential operator of order at most 1.
\end{remark}

\begin{remark}\label{omega_basic_fct}
Notice that if $f\in \Omega^1_b (M)$, \ie $f$ is a basic function, then $df$ is a basic 1-form, hence
\[
f_{\mid j}=f_{j},\quad f_{\mid j \bar{k}}=f_{j \bar{k}}.
\]
\end{remark}

\begin{remark}\label{corespond}
In the above considered local foliated map $\mathcal{U}$, which is homeomorphic to the product $\mathcal{T}\times \mathcal{O}$, any local basic smooth function on $\mathcal{U}$ projects on a smooth function on $\mathcal{T}$. For instance, the local coefficients $g^T_{j\bar{k}}$ of the transverse metric project on the local coefficients $g^{\mathcal{T}}_{j\bar{k}}$ of the {\K} metric on $\mathcal{T}$. The same holds for basic 1-forms and 2-forms, as well for (local) basic vector fields. We saw in \ref{curvature} that the transverse Bott connection and the associated Riemannian curvature tensor can be projected on the local transversal. Thus locally, the transverse Ricci tensor is also projectable on the Ricci tensor on the transversal $\cat$.

As a Riemannian foliation is locally described by a Riemannian submersion, the correspondence is in fact one to one, and any geometric object can be lifted on a (local) basic counterpart (see \cite{Ton} for more details).
\end{remark}

\subsection{Deformations of a Vaisman structure}\label{deformations}
We present in this subsection two deformations of a Vaisman structure that we use in the rest of the paper.
\subsubsection{Deformation of a Vaisman structure using a differential basic function}\label{deform}
In the recent paper \cite{Or-Sl}, a method to deform the structure of a Vaisman manifold using a family of basic functions $\varphi(t)$ is described. More precisely, a family of Vaisman structure $(M, J(t), g(t))$ with $J(0)=J$ and $g(0)=g$ for $t$ small enough is constructed, fixing the canonical foliations and deforming only the complex structure and the metric. The deformation of the complex structure is
\begin{equation}
	J(t)=J-U\otimes d \varphi(t)-V\otimes (d^c_b \varphi(t)\circ J).  \label{J_t}
\end{equation}
The family of anti-Lee 1-forms is
\begin{equation}
 \theta ^{c}(t)=\theta \circ J(t)=\theta ^{c}+d^c_b\varphi(t),
\end{equation}
while the metric is deformed as
\begin{equation}\label{g_t1}
\omega(t)=d\theta ^{c} +d_b d^c_b\varphi(t)-\theta\wedge\theta^c(t)  .
\end{equation}
The transverse metric is deformed in the following way
\begin{equation*}
    \omega^T(t)=\ii(g^T_{j\bar k}+ \varphi(t)_{{j\bar k }})dz^j \wedge d\bar z^k.
\end{equation*}

Thus the basic cohomology group $[\omega^T(t)]$ remains fixed. As a consequence, in this paper we use this method to construct a Vaisman metric for which the transverse part is the solution of the transverse {\K}-Ricci flow (for definition see Section  \ref{transverse_K_R_flow} below).

\subsubsection{Q-homothetic deformations}\label{deform_homot}
We multiply the Lee form $\theta$ by a constant $a>0$, fixing the complex structure $J$ and obtain a new Vaisman metric (this type of deformation is called a {\em 0-type deformation }\cite[Section2]{Be}).
\begin{equation*}
\begin{split}
\tilde g=d(a\theta)+(a\theta^c)\otimes(a\theta^c)+(a\theta)\otimes(a\theta)=ag+(a^2-a)(\theta^c\otimes\theta^c+\theta\otimes\theta).
\end{split}
\end{equation*}

It is easy to see that a homothetic transformation is applied on the transverse metric. A similar transformation for Sasaki manifolds was introduced by Tanno \cite{Tan} (see also \cite[Section 3]{Bo-Ga-Ma}).

We use this transformation in the particular case when the Vaisman metric is transversally {\K}-Einstein (see Section \ref{eta}).

\subsection{Einstein-Weyl manifolds}\label{Einstein-Weyl}

In the final part of preliminary section we briefly recall the notion of Einstein-Weyl manifolds. As the Einstein condition is not invariant with respect to a conformal transformation of the metric, for Vaisman manifolds (and more generally for LCK manifolds) we are interested in finding metric that satisfy the Einstein-Weyl condition.

\begin{definition}
A Riemannian manifold is called \emph{Einstein-Weyl manifold} if it is conformally equivalent to an Einstein manifold.
\end{definition}

We mentioned before that for a Vaisman manifold $M$ the Riemannian universal covering $\tilde M$ is a metric cone over a Sasaki manifold. This covering has a {\K} metric denoted by $\tilde g$. The Levi-Civita connection $D$ on the covering can be projected on the manifold $M$, obtaining a new connection, denoted also by $D$. The link between the connections $\nabla$ and $D$ is expressed by the following relation \cite{Dr-Or}.
\[
D_X Y=\nabla_X Y -\frac 1 2 (\theta(X)Y +\theta(Y)X -g(X,Y)U).
\]

Thus, if on $\tilde g$ we have the relation $\Ric^D=\lambda \tilde g$, where $\Ric^D$ is the Ricci tensor associated to $D$ and $\lambda \in \RR$, then the manifold $M$ is Einstein-Weyl. We say that the pairing $(g,\theta)$ is an \emph{Einstein-Weyl structure}.

\section{The transverse {\K}-Ricci flow on a Vaisman manifold}\label{transverse_K_R_flow}
\subsection{The transverse {\K}-Ricci and Monge-Amp\`ere equations}
Using the foliated structure induced on a Vaisman manifold by the canonical foliation, we present in this subsection the transverse {\K}-Ricci equation and the equivalent Monge-Amp\`ere equation.
\begin{definition}
Let $(M,J,g)$ be a Vaisman manifold and $\omega_0^T$ the transverse K\"ahler form with respect to the canonical foliation. The equation
\begin{equation}
\begin{cases}\label{transv_Ricci_eq}
\frac {\6}{\6 t}\omega^T(t)=-\mathrm{Ric^T}(\omega^T(t)), \\[.1in]
\omega^T(0)=\omega^T_0.
\end{cases}
\end{equation}
is called \emph{the transverse {\K}-Ricci equation}.
\end{definition}

From \eqref{transv_Ricci_eq} we obtain the equation for the cohomology classes.

\begin{equation}\label{eq_classes}
\begin{cases}
\frac {\6}{\6 t}[\omega^T(t)]=-c_1^b(M). \\[.1in]
[\omega^T(0)]=[\omega^T_0].
\end{cases}
\end{equation}
with the solution $[\omega^T(t)]=[\omega^T_0]-t c_1^b(M)$.
Denote by $\cac_M^b$ the transverse {\K} cone:
\[
\cac_M^b :=\{[\alpha]\in H^{1,1}_b (M) \mid \,\mbox{there exists a transverse {\K} form in}\,[\alpha] \}.
\]
As in the K\"ahler case (see \cite[Section 2.1]{To}), $\cac_M^b$ is an open set.
We construct the corresponding Monge-Amp\`ere equation. Let
\[
T_{max}:=\mathrm{sup} \{t>0 \mid [\omega^T_0]-t c_1^b(M)\in \cac_M^b\}.
\]
For $s<T_{max}$ we have  $[\omega^T_0]-s c_1^b(M)\in \cac_M^b$. Let then $\hat{\omega}^T(s) \in [\omega^T_0]-s c_1^b(M)$ be a {\K} form and define
\[
\chi:=\frac 1 {s} \left(\hat{\omega}^T(s)-\omega^T_0\right).
\]
Notice that $\chi \in \cac_M^b$. We define now the basic forms
\[
\hat{\omega}^T(t):=\hat{\omega}^T_0+t \chi, \qquad t\in [0,T_{max}).
\]
As $[\omega^T(t)]=[\hat{\omega}^T(t)]$ for any $t$, $\hat\omega^T(t)$ represents the basic cohomology class of $\omega^T(t)$. Thus there exist a family of smooth basic functions $\{\varphi(t)\}_t$ such that

\begin{equation} \label{omega_varphi}
\omega^T(t)=\hat{\omega}^T(t)+\ii \6_b\bar{\6}_b\varphi(t).
\end{equation}

Let us consider in the following a transverse volume form $\Omega$.
\begin{remark}
Note than since the foliation $\caf$ is totally geodesic, in particular minimal, $\caf$ is transversally orientable and global transversal volume forms can be defined \cite{Ton}.
\end{remark}

On a local map $\mathcal{U}$ with transverse coordinates $z^j,\bar z^j$ there is a local smooth basic function $f$ such that
$$\Omega=f\cdot \ii^n dz^1\wedge d\bar z^1 \wedge\dots\wedge dz^n\wedge d\bar z^n.$$
We define the local two-form $\ii\6_b\bar\6_b \log\Omega:=\ii\6_b\bar\6_b \log f$.
The next result can be proved as in the setting of {\K} manifolds \cite[Subsection 1.4]{We}.

\begin{claim}
$\ii\6_b\bar\6_b \log\Omega$ does not depend on the local foliated map and can be globally defined.
\end{claim}

We consider also the globally defined two-form $\Ric^T\Omega:=-\ii\6_b\bar\6_b \log f$.

Fix now a transverse volume form $\Omega$  such that $\ii \6_b\bar{\6}_b\log \Omega=\chi$.

\medskip

\begin{remark}
A transverse volume form with $\ii \6_b\bar{\6}_b\log \Omega=\chi$ can be constructed as follows. If  $\omega^T_0$ is the  initial transverse {\K} form, consider  the transverse volume form $\Omega_0:=(\omega^T_0)^n$. Since  both $\ii \6_b\bar{\6}_b\log \Omega_0=\ii \6_b\bar{\6}_b\log (\det g^T_{j\bar{k}})$ and $\chi$ represent  $-c_1^b[M]$,  \cite[Proposition 3.5.1]{El Ka} implies the existence of a basic smooth function $F$ such that
\[
\chi-\ii \6_b\bar{\6}_b\log \Omega_0=\ii \6_b\bar{\6}_b F.
\]
Define the transverse volume form $\Omega:=e^F \Omega_0$ (see also \cite{To}).
\end{remark}
\medskip

Adapting the proof from the case of {\K} manifolds (see \eg \cite[Section 3.2]{To}) to our context, we get that the transverse {\K}-Ricci equation \eqref{transv_Ricci_eq} is equivalent to a Monge-A\`mpere equation in the following manner.

\begin{proposition}\label{eq_Ricci_eq_Monge_equiv}
A smooth family of transverse {\K} metrics $\omega^T(t)$ is a solution of equation \eqref{transv_Ricci_eq} for $ t\in [0,\epsilon) $ if and only if there is a smooth family of smooth basic functions $\varphi(t)$,  which satisfies \eqref{omega_varphi} and also the  associated Monge-Amp\`ere equation
\begin{equation}\label{Monge_Ampere}
\begin{cases}
\frac {\6\varphi} {\6t}(t)=\log \frac {(\hat{\omega}^T(t)+\ii \6_b\bar{\6}_b\varphi(t))^n}{\Omega}, \\[.1in]
\varphi(0)=0.
\end{cases}
\end{equation}
\end{proposition}

Notice that  \eqref{Monge_Ampere} can be rewritten as
\begin{equation}\label{M_A}
\begin{cases}
  \frac {\6\varphi} {\6t}(t)=\log \frac {(\hat{\omega}^T(t)+\ii \6_b\bar{\6}_b\varphi(t))^n}{\Omega_0}-F, \\[.1in]
  \varphi(0)=0.
\end{cases}
\end{equation}

\medskip

Let $\omega^T_{\mathcal{U}}$ be the restriction of the transverse {\K} form $\omega^T$ to $\mathcal{U}$, where  $(\mathcal{U}, z_j, \bar z_j, x, y)$ is a foliated chart (Subsection \ref{foliated_map}). Then
\begin{equation*}
%\begin{split}
    \omega^T_{\mathcal{U}}=\ii (\hat g^T_{j\bar{k}}+\varphi_{j \bar{k}}) dz^j \wedge d\bar{z}^k.
%\end{split}
\end{equation*}

In foliated coordinates, the Monge-Amp\`ere equation \eqref{Monge_Ampere} becomes
\begin{equation}\label{M_A_local}
\begin{cases}
\frac {\6 \varphi}{\6 t}(t)=\log \frac{\left(\det(\hat g^T_{j\bar{k}}(t)+\varphi_{j \bar{k}}(t))\right)}
{\left(\det g^T_{j\bar{k}}\right)}+F, \\
\varphi(0)=0,
\end{cases}
\end{equation}
where $\hat g^T_{j\bar{k}}(t)$ are the local coefficients of $\hat{\omega}^T(t)$ (as such, they are basic functions).

\subsection{Short time existence for transverse {\K}-Ricci flow on a Vaisman manifold}\label{short_time_existence}
We present now the method used to prove the existence and uniqueness of the solution of equation \eqref{M_A} (which is equivalent to \eqref{Monge_Ampere}). The equation is only transversally parabolic.
It is important to remark that in order to keep $\omega^T(t)$ invariant along the leaves (thus corresponding to a bundle-like metric) the solution $\varphi(t)$ should be a basic function. As in \cite{Sm-Wa-Zh}, we construct an additional equation which is parabolic also in the leafwise directions and which coincides with \eqref{M_A} when $\varphi(t)$ is basic. This is in fact a frequent technique employed when investigating the transverse geometry of a Riemannian foliation (see \eg the extension of the basic Laplace operator in \cite[Chapter 7]{Ton}). We obtain the existence and uniqueness of the solution using the theory of parabolic equations (see \eg \cite[Chapter 8]{Kr}).
Finally, using maximum principle \cite[Theorem 8.1.2]{Kr}  we prove that the solution $\varphi(t)$ is actually a basic function.

To the family of {\K} forms $\omega^T(t)$ we associate a family of forms $\tilde \omega^T(t)$  which also vanish when applied on the leafwise distribution $T\caf$, but are not necessarily basic. We define the forms in the following manner.
\begin{definition}
For $\varphi(t) \in C^\infty (M)$, denote by $\tilde{\omega}^T(t)$ the differential 2-forms
\[
\begin{split}
\tilde{\omega}^T(X,Y)&:=\hat \omega^T(t)(X,Y) +\ii \6_b \bar{\6}_b \varphi(t)(\pi_Q(X),\pi_Q(Y)) \\
 &=\hat \omega^T(t)(X,Y) +\frac 1 2 \left(\nabla^\omega d\varphi(t)f (J\circ\pi_Q(X),\pi_Q(Y))\right. \\
 &\left.-\nabla^\omega d\varphi(t) (\pi_Q(X),J\circ\pi_Q(Y)) \right).
\end{split}
\]
\end{definition}
for any vector fields $X, Y \in \Gamma(TM)$.

From  \ref{del_b_nonbasic}, $\tilde{\omega}^T$ vanishes when applied to leafwise vector fields. Using \eqref{J_X} and \eqref{derivate}, in the above local foliated chart $\mathcal{U}$,  $\tilde \omega^T$ satisfies

\begin{equation*}
    \begin{split}
    &\tilde \omega^T(X,Y)=-\tilde \omega^T(Y,X), \\
    &\tilde \omega^T(X_j,X_ k)=
    \tilde \omega^T(X_{\bar j},X_{\bar k})=0,\\
    &\tilde \omega^T(X_j,X_{\bar k})=\hat g^T_{j \bar{k}}
    +\varphi_{\mid j \bar{k}}.
    \end{split}
\end{equation*}

Thus the local form of $\tilde \omega^T$ is

\begin{equation}\label{nonumber}
    \tilde{\omega}^T=\ii \left( \hat g^T_{j \bar{k}}
    +\varphi_{\mid j \bar{k}} \right)dz^j\wedge d\bar{z}^k.
\end{equation}

\begin{remark}\label{omega_tilde_omega}
According to \ref{omega_basic_fct}, if $\varphi(t)$ are basic functions, then $\tilde{\omega}^T\equiv \omega^T$.
\end{remark}

We write the new Monge-Amp\`ere equation using the transverse {\K} forms $\tilde{\omega}^T(t)$. We add the corresponding derivates along the leaves, according to the method described above. As the leafwise metric is Euclidean, we obtain the following equation.

\begin{equation}\label{M_A_extended}
\begin{cases}
\frac {\6 \varphi}{\6 t}(t)=\log \frac{\left((\tilde{\omega}^T(t))^n\wedge\theta\wedge\theta^c(t)\right)}
{\left(({\omega}_0^T)^n\wedge\theta\wedge\theta^c\right)} +\frac 1 2 U(U(\varphi(t)))+\frac 1 2 V(V(\varphi(t)))+F, \\
\varphi(0)=0.
\end{cases}
\end{equation}

\begin{claim}\label{M_A_M_A_extended}
If $\varphi$ is a basic function, the equations \eqref{M_A_extended} and \eqref{M_A} coincide.
\end{claim}
\begin{proof}
The result follows using \ref{omega_tilde_omega} and the fact that basic functions are constant along leaves.
\hfill\end{proof}

In the local map $\mathcal{U}$ we obtain the equation
\begin{equation}\label{M_A_extended_local}
\begin{cases}
\frac {\6 \varphi}{\6 t}(t)=\log \frac{\left(\det(\hat g^T_{j\bar{k}}+\varphi_{\mid j \bar{k}})\right)}
{\left(\det g^T_{j\bar{k}}\right)} +\frac 1 2 \frac {\6^2 \varphi}{\6x^2}(t)+\frac 1 2 \frac {\6^2 \varphi}{\6y^2}(t)+F, \\
\varphi(0)=0.
\end{cases}
\end{equation}

Denote by $\tilde \omega(t)$ the family of metrics determined by $\tilde \omega^T(t)$ on the transverse complement $Q$ and the initial Euclidean metric on the leafwise distribution $T\caf$. The operator in \eqref{M_A_extended_local} is  the \emph{complex Laplace} operator
\begin{equation}\label{operator_Laplace}
\tilde \Delta=\frac 1 2 \mathrm{tr}_{\tilde \omega(t)} \nabla^\omega d.
\end{equation}
 From  \ref{leading_symbol} we see that this differential operator has the same leading symbol as the Laplace operator
$\Delta_{\tilde \omega(t)}=\frac 1 2 \mathrm{tr}_{\tilde \omega(t)} \nabla^{\tilde \omega(t)} d$. Thus $\tilde \Delta$ is also an elliptic second order differential operator on $M$. The existence and uniqueness of the solution of equation \eqref{M_A_extended_local} on $\mathcal{U}$ can be obtained using the theory of parabolic equation with variable coefficients (see \eg \cite[Chapter 8]{Kr}).
\begin{remark}
For convenience, in the rest of the paper we write $\Delta=\Delta_{\tilde \omega(t)}$.
\end{remark}

The solution also has the property that $\hat g^T_{j\bar k}+\varphi_{\mid j\bar k}$ are the coefficients of a transverse metric tensor, positively defined on the transverse distribution $Q$. The proof is similar to the case of {\K} manifolds, and we refer to \cite[Section 3]{To}. We obtain a solution of class $C^\infty$ of the equation \eqref{M_A_extended_local} on $\mathcal U \times [0,\epsilon_{\mathcal U})$, for $\epsilon_{\mathcal U}>0$ depending on $\mathcal U$. As the manifold $M$ is compact, we use a finite covering to get the existence and uniqueness of the global solution of the equation \eqref{M_A_extended} on $M\times [0, \epsilon)$.

Next we prove that the solution $\varphi(t)$ is a basic function for any $t\in [0, \epsilon)$. The property is local, so we work on the local foliated map $\mathcal U$.

For the computations we use the next results.
\begin{claim}
The following Lie brackets vanish:
\begin{equation}\label{brackets_vanish}
    \begin{split}
    &\left[\frac{\6}{\6 x},X_j\right]= \left[\frac{\6}{\6 x},\bar X_j\right]
    = \left[\frac{\6}{\6 x},\nabla^\omega_{X_j}\bar X_k\right]=0, \\
    &\left[\frac{\6}{\6 y},X_j\right]= \left[\frac{\6}{\6 y},\bar X_j\right]
    = \left[\frac{\6}{\6 y},\nabla^\omega_{X_j}\bar X_k\right]=0.
    \end{split}
\end{equation}
\end{claim}

\begin{proof}
In the local chart $\mathcal U$ considering the relations \eqref{X_j}, as $h_j$, $h_{\bar j}\in \Omega_b(M)$, we have

\[
\left[\frac{\6}{\6 x},X_j\right]= \left[\frac{\6}{\6 x},\bar X_j\right]= 0,
\]
The similar result holds for the leafwise vector field $\frac{\6}{\6 y}$.

Clearly the coefficients $g_{i \bar j}$ of the  metric $g$ are basic functions, $g_{i \bar j} \in \Omega_b(M)$.

We get
\[
g\left(X_j, \frac{\6}{\6 x} \right)=g\left(\frac{\6}{\6 z^j}, \frac{\6}{\6 x} \right)-
\ii \left(\frac{\6}{\6 x}, \frac{\6}{\6 x} \right)
=g_{jx}-\ii h_j=0,
\]
where $g_{jx}=g\left(\frac{\6}{\6 z^j}, \frac{\6}{\6 x} \right)$. As the restriction of the metric to the leaves is Euclidean, we get that all coefficients of the metric are basic functions. Thus all Levi-Civita coefficients are basic functions, and

\begin{equation*}
    \left[\frac{\6}{\6 x},\nabla^\omega_{X_j}\bar X_k\right]=0.
\end{equation*}
For the vector field $\frac{\6}{\6 y}$ the proof is similar.
\hfill\end{proof}

\begin{claim}
The following relations hold
\begin{equation}\label{comut_Hesse}
 % \begin{split}
    \frac{\6}{\6 x}(\varphi_{\mid j \bar k})(t)=\frac{\6}{\6 x}(\varphi)_{\mid j \bar k}(t), \quad
    \frac{\6}{\6 y}(\varphi_{\mid j \bar k})(t)=\frac{\6}{\6 y}(\varphi)_{\mid j \bar k}(t).
 % \end{split}
\end{equation}
\end{claim}

\begin{proof}
The commutation of the derivatives follows directly from \eqref{brackets_vanish}.
\hfill\end{proof}

\begin{claim}
If $\varphi$ is a solution of the equation \eqref{M_A_extended} on $M\times [0,\epsilon)$, then it satisfies also the equation
\begin{equation} \label{M_A_varphi_basic_global}
    \begin{split}
    \frac{\6}{\6 t} (U(\varphi)^2)&=2 U(\varphi)U\left(\log \frac{\left((\tilde{\omega}^T(t))^n\wedge\theta\wedge\theta^c\right)}
{\left(({\omega}_0^T)^n\wedge\theta\wedge\theta^c\right)}  \right)    \\
   & +\frac 1 2 U^2(U(\varphi)^2)+ \frac 1 2 V^2(U(\varphi)^2)-(U^2(\varphi))^2-(U(V(\varphi)))^2.
   \end{split}
\end{equation}
\end{claim}
\begin{proof} \textbf{Step 1: We use equation \eqref{M_A_extended} to compute the following.}
\begin{equation*}
    \begin{split}
    \frac{\6}{\6 t} (U(\varphi)^2)&
    =2U(\varphi)\frac{\6}{\6 t}(U(\varphi))
    =2U(\varphi)U\left(\frac{\6 \varphi}{\6 t}\right) \\
    &=2 U(\varphi)U\left(\log \frac{\left((\tilde{\omega}^T(t))^n\wedge\theta\wedge\theta^c\right)}
{\left(({\omega}_0^T)^n\wedge\theta\wedge\theta^c\right)}  \right)
    +U(\varphi)U^2(U(\varphi)) \\
   & +U(\varphi)V^2(U(\varphi))+U(\varphi)U\left(F\right).
    \end{split}
\end{equation*}

\textbf{Step 2: For the last three terms we use that $F\in \Omega_b(M)$ and $[U,V]=0$}. The following relations hold.

\begin{equation*}
    \begin{split}
    U(\varphi)U^2(U(\varphi))&=\frac 1 2 U^2(U(\varphi)^2)-(U^2(\varphi))^2, \\
    U(\varphi)V^2(U(\varphi))&=\frac 1 2 V^2(U(\varphi)^2)-(U(V(\varphi)))^2, \\
    U(F)&=0.
    \end{split}
\end{equation*}
Thus we obtain equation \eqref{M_A_varphi_basic_global}.
\hfill\end{proof}

For convenience we denote in the sequel $U(\varphi)$ by $u$.
We now express equation \eqref{M_A_varphi_basic_global} in the local chart $\left(\mathcal{U}, z^1,\ldots, z^n, \bar{z}^1, \ldots, \bar{z}^n,x,y\right)$.

\begin{lemma}\label{M_A_varphi_basic_local1}
In local coordinates the equation \eqref{M_A_varphi_basic_global} becomes
\begin{equation}\label{M_A_varphi_basic_local}
    \begin{split}
    \frac{\6}{\6 t}u^2&=({\tilde g^T})^{j \bar k} u^2_{\mid j \bar k}
    +\frac 1 2 \frac{{\6}^2 u^2}{\6 x^2}+\frac 1 2 \frac{{\6}^2 u^2}{\6 y^2}
    -({\tilde g^T})^{j \bar k}u_{\mid j}u_{\mid \bar k} \\
    &-\left(\frac{\6 u}{\6 x}\right)^2-\left(\frac{\6 u}{\6 y}\right)^2,
    \end{split}
\end{equation}
where $({\tilde g^T})^{j \bar k}$ are the coefficients of the inverse matrix of $\tilde g_{j \bar k}^T=\hat g_{j \bar k}^T+\varphi_{\mid j \bar k}$.
\end{lemma}

\begin{proof}
From \eqref{M_A_varphi_basic_global} we get on $\mathcal U$

\begin{equation*}
    \frac{\6}{\6 t}u^2=2u \frac{\6}{\6 x}\left(\log \frac{\left(\det(\hat g^T_{j\bar{k}}+\varphi_{\mid j \bar{k}})\right)}
{\left(\det g^T_{j\bar{k}}\right)}\right)
+\frac 1 2 \frac{{\6}^2 u^2}{\6 x^2}+\frac 1 2 \frac{{\6}^2 u^2}{\6 y^2}
-2\left(\frac{\6 u}{\6 x}\right)^2-2\left(\frac{\6 u}{\6 y}\right)^2.
\end{equation*}
Thus, using \eqref{comut_Hesse} we derive
\begin{equation*}
\begin{split}
\frac{\6}{\6 x}\left(\log \frac{\left(\det(\hat g^T_{j\bar{k}}+\varphi_{\mid j \bar{k}})\right)}
{\left(\det g^T_{j\bar{k}}\right)}\right)& =\frac{\6}{\6 x}\left(\log \left(\det(\hat g^T_{j\bar{k}}+\varphi_{\mid j \bar{k}})\right) \right) \\
&=({\tilde g^T})^{j \bar k}\frac{\6}{\6 x}(\hat g^T_{j \bar k}+\varphi_{\mid j \bar k})
=({\tilde g^T})^{j \bar k}
{u}_{\mid j \bar k}
\end{split}
\end{equation*}
Finally $2u \cdot u_{\mid j \bar k}=u^2_{\mid j \bar k}-2u_{\mid j}u_{\mid \bar k}$, and the equation \eqref{M_A_varphi_basic_local} is obtained.
\hfill\end{proof}

In the following denote for convenience the right part of \eqref{M_A_varphi_basic_local} by $Lu$.
\begin{remark}\label{Lu_eliptic}
We remark that the terms of order 2 of $Lu$ are in fact the complex Laplace operator written in \eqref{operator_Laplace}. Thus the equation is of parabolic type.
\end{remark}

We prove that if $\varphi $ is a solution of equation \eqref{M_A_varphi_basic_global} on $M\times[0,\epsilon)$, then $u=0$. We achieve this using the maximum principle for this type of equation \cite[Theorem 8.1.2]{Kr}.

\begin{claim}\label{max_priciple}
Assume that $\varphi$ is a solution of equation \eqref{M_A_varphi_basic_global} on $M\times[0,\epsilon)$. Thus if $u=0$ for $t=0$, then this property holds for any $t \in [0,\epsilon)$.
\end{claim}

\begin{proof}
\textbf{Step. 1: Define an auxiliary function $v=u^2-\frac \gamma {\epsilon^\prime -t}$}. Here $\gamma>0$ is a real number, and the function is defined on the domain $\bar D$, where $D=M\times(0,\epsilon^\prime)$, $\epsilon^\prime \in [0,\epsilon)$.

\textbf{Step. 2: We prove that the maximum of $v$ is attained on the \emph{parabolic} boundary $M\times\{0\}$. } As $\varphi$ is a smooth function on $\bar D=M\times[0,\epsilon^\prime]$, the function $v=u^2-\frac \gamma {\epsilon^\prime -t}$ attains its maximum at a point $p_\gamma \in \bar D$. As $v\rightarrow -\infty$ when $t\rightarrow \epsilon^\prime$, $p_\gamma$ it is not on the upper lid $M\times \{\epsilon^\prime \}$. Suppose that $p_\gamma \in D$. Let $\left(\mathcal{U}, z^1,\ldots, z^n, \bar{z}^1, \ldots, \bar{z}^n,x,y\right)$ be a foliated map around $p_\gamma$. Then, according to \ref{M_A_varphi_basic_local1}, the equation has the form \eqref{M_A_varphi_basic_local}. In the following consider $t$ fixed. At the point $p_\gamma$ of coordinates $(z^i_{p_\gamma},\bar z^i_{p_\gamma}, x_{p_\gamma}, y_{p_\gamma}, t_{p_\gamma})$, the map
$(z^i,\bar z^i, x, y) \longmapsto v$ has a local maximum. Hence $u^2$ has also a local maximum at $p_\gamma$, while $u$ has a local extremum.

Using \ref{Lu_eliptic} we also obtain
\begin{equation*}
  % \begin{split}
   \left(u_{\mid j}\right)_{p_\gamma}=\left(u_{\mid \bar k}\right)_{p_\gamma}=0, \quad
   \left(Lu\right)_{p_\gamma}\le 0.
  % \end{split}
\end{equation*}
for $1\le j,k\le n$.

Now we fix $z^i,\bar z^i, x, y$ and consider the map $t \longmapsto v$. At ${p_\gamma}$ we get

\begin{equation*}
    \left(\frac{\6v}{\6t} \right)_{p_\gamma}=\left(\frac{\6u^2}{\6t}\right)_{p_\gamma}
    -\frac{\gamma}{(\epsilon^\prime-t_{p_\gamma})^2} =0.
\end{equation*}
Thus
\begin{equation*}
    0=\left(Lu-\frac{\6 u^2}{\6 t}\right)_{p_\gamma}\le -\frac{\gamma}{(\epsilon^\prime-t_{p_\gamma})^2}<0.
\end{equation*}
The contradiction implies that the maximum of $v$ is attained for $t=0$.

\textbf{Step 3: Obtain an estimation for the solution}. As a consequence $v\le0$, and as $\gamma>0$ is arbitrary, we get that $u=0$ on $\bar D$.

Finally, $\epsilon^\prime \in [0,\epsilon)$ can be taken arbitrary and the result follows.
\hfill\end{proof}

\begin{lemma}\label{varphi_basic}
The solution $\varphi\in M\times [0,\epsilon)$ of the equation \eqref{M_A_extended} is a basic function.
\end{lemma}
\begin{proof}
From \ref{max_priciple} we obtain that $U(\varphi)=0$. Using a similar argument we obtain $V(\varphi)=0$. As the leafwise distribution $T\caf$ is generated by the vector fields $U$ and $V$, we obtain the conclusion.
\hfill\end{proof}
\begin{proposition}\label{solution_Monge_short}
 The transverse Monge-Amp\`ere equation \eqref{Monge_Ampere} has a solution $\varphi$ for a short time $t\in[0,\epsilon)$, which is a basic function.
\end{proposition}
\begin{proof}
The result is obtained from \ref{varphi_basic} and \ref{M_A_M_A_extended}.
\hfill\end{proof}

\begin{remark}
If the basic Chern class $c_1^b(M)$ and the cohomology class of the initial {\K} form $\omega^T_0$ are connected by the relation
\begin{equation}\label{condition}
    c_1^b(M)=k[\omega^T_0],
\end{equation}
where, using a rescaling, $k$ can be taken as $-1$, $0$ or $1$, then $[\omega^T(t)]$ is fixed and $\omega^T(t)=\omega^T_0+\ii \6\bar \6 \varphi(t)$.
\end{remark}

Finally, we use \ref{eq_Ricci_eq_Monge_equiv} and the method of deformation of a Vaisman structure  with basic functions from Section \ref{deform} to obtain the following theorem.

\begin{theorem}
On a Vaisman manifold $M$ the transverse {\K}-Ricci equation \eqref{transv_Ricci_eq} has a solution for $t\in [0,\epsilon)$, for a small $\epsilon>0$. Furthermore, if the condition \eqref{condition} is satisfied, the transverse {\K}-Ricci flow can be associated to a family of Vaisman structures defined on $M$.
\end{theorem}

In the following, for the maximal existence time, the main reference source is \cite{Be-He-Ve};  the result is obtained in a similar manner to the case of compact {\K} manifold, using a local one to one correspondence between transverse mathematical objects and their projections on a local transversal (see also \ref{corespond} ). For global estimates, the authors devise a min-max principle for basic maps (see \cite[Proposition 6.9]{Be-He-Ve}).  In our particular frame, the result obtained in the above referred paper for general Riemannian foliations becomes:

\begin{theorem}
Let $M$ be a Vaisman compact manifold and let $\omega^T(t)$ be the transverse {\K}-Ricci flow with respect to canonical foliation, described by relations \eqref{transv_Ricci_eq}. Then $\omega^T(t)$ is defined on the maximal time interval $[0, T_{max})$, where $T_{max}\le \infty$ is defined as
\begin{equation*}
  T_{max}=\sup\{T>0\mid [\omega^T_0]-Tc_1^b(M)\in \mathcal{C}^b_M\}.
\end{equation*}
\end{theorem}

Assume now that $c_1^b(M)<0$, \ie there is a transverse metric $\alpha^T$ such that $-\alpha^T\in c_1^b(M)$. Thus, from \eqref{eq_classes} we obtain that $[\omega^T(t)]\in \mathcal{C}^b_M$, and the transverse {\K}-Ricci flow is defined for any $t \in[0,\infty)$. We investigate in what follows the long time convergence of $\omega^T(t)$.

From the equation \eqref{eq_classes}, under the above assumptions $[\omega^T(t)]$ does not converges when $t \rightarrow \infty$. Thus, as in the case of {\K} manifolds, we "rescale" the transverse {\K}-Ricci equation. More precisely, we consider the equation
\begin{equation}
\begin{cases}\label{rescaled_transv_Ricci_eq}
\frac {\6}{\6 t}\omega^T(t)=-\mathrm{Ric}(\omega^T(t))-\omega^T(t), \\
\omega^T(0)=\omega^T_0.
\end{cases}
\end{equation}
We rescale the flow as follows. Take  $\bar\omega^T(s)=e^t\cdot\omega^T(t)$, where $s=e^t-1$. Then
\begin{equation*}
   \frac {\6}{\6 t}\omega^T(t)=\frac{\6}{\6 t}\bar \omega^T(s)-\frac{\bar \omega^T(s)}{s+1},\, \mbox{and}\, \Ric(\omega^T(t))+\omega^T(t)=\Ric(\bar\omega^T(s))+\frac{\bar \omega^T(s)}{s+1}.
\end{equation*}
From \eqref{rescaled_transv_Ricci_eq} we obtain $\frac{\6}{\6 t}\bar \omega^T(s)=-\Ric(\bar\omega^T(s))$, so $\omega^T(t)$ verifies the rescaled equation \eqref{rescaled_transv_Ricci_eq} if and only if $\bar\omega^T(s)$ verifies the transverse {\K}-Ricci equation \eqref{transv_Ricci_eq}. Applying the results of \cite{Be-He-Ve} in the particular framework of Vaisman compact manifolds, we get:

\begin{theorem}\label{theorem}
Assume that on a Vaisman manifold $M$  the basic Chern class satisfies $c_1^b(M)<0$. Then the rescaled {\K}-Ricci equation \eqref{rescaled_transv_Ricci_eq}
has solution for $t\in[0,\infty)$. For $t \rightarrow \infty$, $\omega^T(t)$ converges to a {\K} transverse metric $\omega^T_\infty$ which satisfies

\[
\Ric(\omega^T_\infty)=-\omega^T_\infty.
\]
The limit $\omega^T_\infty$ does not depend on initial metric $\omega^T_0$.

Assume, furthermore, that basic Chern class satisfies $c^b_1(M)=-[\omega^T_0]$. Thus the transverse {\K}-Ricci flow is associated to a families of Vaisman metrics defined on the manifold $M$.
\end{theorem}

For the last assertion we use the deformation of a Vaisman structure described in Section \ref{deform}.
\begin{remark}
As a consequence, the above conditions imposed on the basic Chern class $c_1^b(M)$ are enough for the existence of a Vaisman manifold with transverse {\K}-Einstein metric.
\end{remark}

\subsection{An example}
In the final part of this section, we present an example. In accordance with \cite[Section 2]{Or-Ve1}, us consider a projective manifold $\mathbf{Q}$. Assume the existence of a {\K}-Einstein metric $\omega_\mathbf{Q}$, with $\Ric_\mathbf{Q}(\omega_\mathbf{Q})=-\omega_\mathbf{Q}$. Consider $L$ a positive line bundle on $\mathbf{Q}$. Denote by $\mathrm{Tot}^\circ(L)$ the total space of all non-zero vectors in $L$. If we denote by $\|v\|$ the length of a vector $v$, then $\tilde \omega=d d^c \|v\|^2$ is a {\K} potential on $\mathrm{Tot}^\circ(L)$.

Take  $q \in\RR^{>0}$, and choose $\sigma_q:\mathrm{Tot}^\circ(L) \rightarrow \mathrm{Tot}^\circ(L)$ to  map a vector $v$ in $qv$. Consider also the $\ZZ$ action on $\mathrm{Tot}^\circ(L)$ determined by $\sigma_q$. $\omega:=\frac{\tilde \omega}{ \|v\|^2}$ is projectable on the quotient manifold $\mathrm{Tot}^\circ(L)/\ZZ$, which is thus equipped with a natural Vaisman structure.

The leaves of the canonical foliation associated to the Vaisman manifold are in fact the fibers of $L$. We obtain a regular foliation  (\ie the leaves are compact and the projection to the leaf space is a smooth submersion).

Denote the transverse metric on the Vaisman foliation by $\omega^T_{KE}$. Let  $\Ric^T$ be the transverse Ricci operator. Locally, $\omega^T_{KE}$ and $\Ric^T$ project on $\omega_\mathbf{Q}$ and $\Ric_\mathbf{Q}(\omega_\mathbf{Q})$. We obtain
\begin{equation*}
   \Ric^T(\omega^T)=-\omega^T,
\end{equation*}
so the metric is \emph{transversally {\K}-Einstein}.

As in Section \ref{deform}, we use a basic function to deform the Vaisman metric to a new Vaisman metric $\omega_0$.
Hence, if $\omega^T_0$ is the corresponding transverse metric, we have
\begin{equation*}
    c^b_1(M)=-[\omega^T_{KE}]=-[\omega^T_0].
\end{equation*}

We can thus apply \ref{theorem} and obtain a transverse {\K}-Ricci flow.  This flow deforms the Vaisman metric $\omega_0$ into a new Vaisman metric which is transversally {\K}-Einstein. Due to the uniqueness of such metric, it should be in fact the initial Vaisman metric (with transverse component $\omega^T_{KE}$).

\begin{remark}
A Sasaki manifold with transverse Einstein metric is called $\eta$-Einstein (see \eg \cite{Bo-Ga-Ma}). Equivalently, a Sasaki manifold $(M,\phi,\xi,\eta,g)$ is $\eta$-Einstein if and only if there are $\lambda,\beta \in C^\infty (M)$ such that $Ric = \lambda g + \beta\eta\otimes\eta $.

It is clear that the above example represents the counterpart of this type of metrics in the setting of Vaisman manifolds.
\end{remark}

\section{Quasi-Einstein metrics and Einstein-Weyl structures on Vaisman manifolds}\label{eta}

In this section, we focuss our study to the investigation of the class of quasi-Einstein metrics first considered in LCK setting  by Goldberg and Vaisman in \cite{Go-Va}. Recall that a
LCK manifold $(M,J,g)$ is said to be a quasi-Einstein space if Ricci curvature $Ric$ of $M$ is given by $\Ric = \lambda g + \beta\theta\otimes\theta $, for some  $\lambda$ and $\beta \in C^\infty (M)$, where $\theta$ is the Lee form of $M$. As in the previous sections,  we work in the setting represented by Vaisman manifolds. In the following,  we use the fact that, locally, a Vaisman manifold can be regarded as a codimension 1 foliation, with leaves having a local Sasaki structure. More precisely, as we assumed $\|\theta\|=\|\theta^c\|=1$,  the Sasaki metric $g_S$ in the leaves is obtained by multiplying the metric induced by the Vaisman metric $g$ with the constant $1/4$. Thus we have the relations
\begin{equation*}
 \theta^c=2\eta, \quad V=\frac \xi 2.
\end{equation*}

 We use the standard relation in Sasaki geometry

 \begin{equation}\label{Ric_Ric_T}
 \Ric_S^T(X,Y)=\Ric_S(X,Y)+2g_S(X,Y),
 \end{equation}
which is valid for any transverse vector fields $X$ and $Y$ (see \eg \cite{Sm-Wa-Zh}). The multiplication of the metric by a constant do not change the Ricci tensor. Also, as the Lee vector field $U$ is parallel,
\begin{equation}
R(U,X,Y,Z)=0,
\end{equation}
for any tangent vector fields $X$, $Y$ and $Z$. We obtain
\begin{equation}\label{R_S_R}
\begin{cases}
\Ric^T(X,Y)=\Ric_S^T(X,Y), \\
\Ric(X,Y)=\Ric_S(X,Y),
\end{cases}
\end{equation}
for any $X, Y \in \Gamma(TM)$.
Thus on the Vaisman manifold $M$ the Ricci tensor is described by the relations
\begin{equation}\label{Ricci_S_R}
\begin{cases}
\Ric(V,V)=\frac n 2, \\
\Ric(V,X)=0 \quad \mbox{for any}\quad X \in \Gamma(Q), \\
\Ric(U,X)=0 \quad \mbox{for any}\quad X \in \Gamma(TM). \\
\end{cases}
\end{equation}

From \eqref{Ricci_S_R} the following result is straightforward.
\begin{claim}
If $\Ric=\lambda g+\alpha \theta^c\otimes\theta^c+\beta \theta\otimes\theta $, with $\lambda$, $\alpha$ and $\beta \in C^\infty (M)$, then $\lambda+\alpha=\frac n 2$ and $\beta=-\lambda$.
\end{claim}

We also prove
\begin{claim}
If $n\ge 3$, then the function $\lambda$ (and thus $\alpha$) is a constant.
\end{claim}
\begin{proof}
As for the framework of Riemannian manifolds (see \cite[Ch.1]{Bes}), we apply the following differential operator on the symmetrical Ricci tensor
\begin{equation*}
\delta \Ric(X)=\sum_{j=1}^n \nabla_{E_j}\Ric(E_j, X),
\end{equation*}
for any orthonormal basis $\{E_j\}$ and $X \in\Gamma(TM)$. If $s$ is the scalar curvature, then $ds=2\delta \Ric$. Similar to the case of Sasaki manifolds (see \cite[Ch. 11]{Bo-Ga}), we get

\begin{equation*}
ds=2d\lambda-2V(\lambda)\theta^c-2U(\lambda)\theta=2 d_b \lambda.
\end{equation*}

Thus $ds(U)=ds(V)=0$, and $s$ is a basic function, $d_b s=d s$. The same holds for $\lambda$. As
\begin{equation*}
s=\sum_{j=1}^{2n}\Ric(e_j, e_j)+\Ric(U,U)+\Ric(V,V)=2n\lambda+\frac n 2,
\end{equation*}
the smooth functions $\lambda $ and $s$ are also related by the formula $d_b s=2n d_b \lambda$. The conclusion follows.
\hfill \end{proof}

Using also \eqref{R_S_R}, the counterpart  of  \eqref{Ric_Ric_T} for Vaisman manifolds is
 \begin{equation}\label{Ric_Ric_T_Vaisman}
 \Ric^T(X,Y)=\Ric(X,Y)+\frac 1 2g(X,Y),
 \end{equation}
Assume now that the Vaisman manifold is transversally Einstein, with respect to the constant $\lambda+\frac 1 2$. We apply the transverse homothetic transform described in Subsection \ref{deform}:

\begin{equation*}
\tilde g=ag+(a^2-a)(\theta^c\otimes\theta^c+\theta\otimes\theta).
\end{equation*}
For $X$, $Y \in \Gamma (Q)$, we get
\begin{equation}
\begin{split}
\Ric_{\tilde g}(X,Y)&=\Ric_{\tilde g}^T(X,Y)-\frac 1 2 \tilde g(X,Y) \\
&=\Ric(X,Y)+\frac 1 2 (1-a) g(X,Y) \\
&=\frac{2\lambda+a-1}{2a}\tilde g(X,Y).
\end{split}
\end{equation}
If $\lambda>-\frac 1 2$, then taking $a=\frac{2\lambda+1}{n+1}$ we obtain
\begin{equation*}
\Ric_{\tilde g}(X,Y)=\frac n 2 \tilde g(X,Y).
\end{equation*}
We proved the following result.
\begin{proposition}\label{prop_}
Assume that a Vaisman manifold is transversally {\K}-Einstein with respect to the canonical foliation, with the constant $\lambda +\frac 1 2>0$. Then, by an appropriate transverse $Q$-homothetic deformation, one can transform the Vaisman metric into a transverse {\K}-Einstein  metric, this time $Q$ being the distribution complementary to the Lee vector field.
\end{proposition}

\begin{remark}\label{remark_}
Using the definition from \cite{Ch}, we  actually showed that if the Vaisman manifold is generalized quasi Einstein  with $\Ric=\lambda g+(\frac n 2-\lambda)\theta^c\otimes\theta^c-\lambda\theta\otimes\theta$ and $\lambda>-\frac 1 2$, then we can deform the metric to a Vaisman metric $\tilde g$ which is quasi Einstein, $\Ric_{\tilde g}=\frac n 2 \tilde g-\frac n 2 \tilde\theta\otimes \tilde\theta$, with $\tilde \theta=a\theta$.
\end{remark}

\begin{claim}\label{claim_}
Assume that $(M, g, \theta)$ is a Vaisman structure with $\| \theta\|=1$. Then $(g,\theta)$ is Einstein-Weyl if and only if the Vaisman metric is quasi Einstein.
\end{claim}

\begin{proof}
Using \cite{Hi}, we have
\begin{equation*}
\begin{split}
\Ric^D(X,Y)&=\Ric(X,Y)+\frac 1 2\nabla_Y\theta (X)-\frac {2n+1}{2}\nabla_X\theta (Y)+\frac n 2\theta(X)\theta(Y)  \\
                 &+\frac 1 2(\delta\theta-n\|\theta\|^2)g(X,Y) \\
                 &=\Ric(X,Y)+\frac n 2 {\theta(X)}\cdot  {\theta(Y)}-\frac n 2 g(X,Y),
\end{split}
\end{equation*}
as $\theta$ is parallel, and $\| \theta\|=1$. On the Vaisman manifold $M$,   $(g, \theta)$ is an Einstein-Weyl structure if and only if $\Ric^D\equiv 0$ \cite{Ga}, and the claim follows.
\hfill\end{proof}

From \ref{prop_} and \ref{claim_} we obtain the following result.

\begin{proposition}\label{prop1}
If $(M, g, \theta)$ is a Vaisman structure with $\| \theta\|=1$,  which is transversally {\K}-Einstein with respect to the canonical foliation, with the constant $\lambda+\frac 1 2>0$, then one  can find a second Vaisman structure $(M, \tilde g, \tilde \theta)$ for which  $(\tilde g, \tilde \theta)$ is an Einstein-Weyl structure.
\end{proposition}

\begin{example}
A Vaisman manifold $(M, J, g)$ is said to be a $\mathcal{P}_0K$ manifold if the curvature tensor of the Weyl connection vanishes identically.
If the real dimension of $(M, J, g)$ is $2n+2$, then a direct computation shows that the Ricci tensor of a $\mathcal{P}_0K$ manifold is given by (see \cite{Ia-Ma-Or})
\begin{equation}\label{eqf}
\Ric=\frac{n\|\theta\|^2}{2}g- \frac{n}{2}\theta\otimes\theta.
\end{equation}
Hence, any $\mathcal{P}_0K$ manifold is quasi-Einstein with $\lambda=\frac{n\|\theta\|^2}{2}$ and $\beta=-\frac{n}{2}$.
 Note that $\lambda$ and $\beta$ in \eqref{eqf} are both constant, since $\|\theta\|$ is known to be constant on Vaisman manifolds. In particular, if  $\|\theta\|=1$, then $(g,\theta)$ is an Einstein-Weyl structure. Concerning the geometry of $\mathcal{P}_0K$ spaces, it is known that a $\mathcal{P}_0K$ manifold is locally analytically homothetic to a Hopf manifold and, according to \cite[Proposition 2.7]{Ia-Ma-Or}, such  manifolds are locally symmetric spaces.

Finally, we want to emphasize that the name of $\mathcal{P}_0K$ manifold originates from the old name of Vaisman manifolds, which was  $\mathcal{P}K$ manifold. Recall that the study of Vaisman manifolds was initiated in \cite{Va0} under the name of generalized Hopf manifolds (GHM) or $PK$-manifolds, and such spaces were investigated under this name for a long period of time. However, in the monograph \cite{Dr-Or}, the authors explained why the old name is not appropriate, proposing to call a LCK manifold with parallel Lee form a Vaisman manifold. This name was later unanimously accepted by authors working in the field. Hence, the name of $\mathcal{P}_0K$ manifold was originally chosen to highlight that it is a particular type of $\mathcal{P}K$ manifold, \emph{i.e.} a Vaisman manifold or a GHM manifold, satisfying a nullity condition, that is the curvature tensor of the Weyl connection vanishes identically.
\end{example}

\subsection*{Acknowledgement} The authors would like to thank Professor Luigi Vezzoni for bringing the article \cite{Be-He-Ve} to their attention and for useful comments on the first draft of this paper. The authors are also extremely grateful to Professor Liviu Ornea for the useful suggestions and remarks that led to the improvement of the article.

\hfill

{\small

	\noindent {\sc Vladimir Slesar\\
		 University ``Politehnica'' of Bucharest, Faculty of Applied Sciences, \\
			313 Splaiul Independen\c{t}ei,
			060042 Bucharest, Romania}, and:\\
		{\sc University of Bucharest, Faculty of Mathematics and Informatics, \\
		Research Center in Geometry,  Topology and Algebra, \\
		14 Academiei str., 70109 Bucharest, Romania\\
	\tt vladimir.slesar@upb.ro}

	\hfill
	
	\noindent {\sc Gabriel-Eduard V\^{\i}lcu\\
	{\sc University of Bucharest, Faculty of Mathematics and Informatics, \\
		Research Center in Geometry,  Topology and Algebra, \\
		14 Academiei str., 70109 Bucharest, Romania}, and:\\
    {\sc University ``Politehnica'' of Bucharest, Faculty of Applied Sciences\\
    Department of Mathematics and Informatics,\\ 313 Splaiul Independen\c{t}ei,
			060042 Bucharest, Romania}, and:\\
    {\sc Department of Cybernetics, Economic Informatics, Finance and Accountancy, \\
		Petroleum-Gas University of Ploie\c sti, \\
		Bd. Bucure\c sti, Nr. 39, 100680 Ploie\c sti, Romania}\\
		\tt gvilcu@upg-ploiesti.ro}

\end{document}